\documentclass[11pt]{amsart} 

\usepackage[tableposition=top]{caption}
\usepackage{amssymb,amsmath,latexsym,enumerate}
\usepackage{graphicx,bbm,mathptmx,microtype,cite}
\usepackage{hyperref}
\usepackage{color}

\newtheorem{theorem}{Theorem} % [section] \numberwithin{equation}{section}
\newtheorem*{lrc}{Lonely Runner Conjecture} 

\newtheorem{corollary}[theorem]{Corollary}
\newtheorem{proposition}[theorem]{Proposition}

\newtheorem{remark}[theorem]{Remark}

\newtheorem{question}[theorem]{Question}

\renewcommand\emptyset{\varnothing}
\renewcommand\phi{\varphi}

\newcommand{\lext}[1]{\mathbf{#1}}

\newcommand\conv{\operatorname{conv}} 
\newcommand\inter{\operatorname{int}} 
 
\newcommand\ca{\operatorname{ca}} 
\newcommand\rec{\operatorname{rec}}

\DeclareMathOperator*{\vol}{vol}

\newcommand\run{\operatorname{Run}} 
 
\newcommand\rab{\operatorname{Rab}} 
 
\newcommand\gap{\operatorname{gap}}

\newcommand\ZZ{\mathbb{Z}}
\newcommand\RR{\mathbb{R}}
\newcommand\NN{\mathbb{Z}_{ > 0 }}
\newcommand\QQ{\mathbb{Q}}

\newcommand\cO{\mathcal{O}}

\newcommand\Zon{\operatorname{\mathcal{Z}}}

\newcommand\bP{\mathbf{P}}

\newcommand\bc{\mathbf{c}}
\newcommand\be{\mathbf{e}}
\newcommand\bg{\mathbf{g}}
\newcommand\bm{\mathbf{m}}
\newcommand\bn{\mathbf{n}}
\newcommand\bp{\mathbf{p}}

\newcommand\br{\mathbf{r}}

\newcommand\bt{\mathbf{t}}

\newcommand\bv{\mathbf{v}}
\newcommand\bw{\mathbf{w}}
\newcommand\bx{\mathbf{x}}
\newcommand\by{\mathbf{y}}

\newcommand\bzero{\mathbf{0}}
\newcommand\bone{\mathbf{1}}

\begin{document}

\title{Deep lattice points in zonotopes, lonely runners, and lonely rabbits}

\author{Matthias Beck}
\address{Department of Mathematics\\
         San Francisco State University\\
         San Francisco, CA 94132\\
         U.S.A.}
\email{becksfsu@gmail.com}

\author{Matthias Schymura}
\address{Institut f\"ur Mathematik\\
  Universit\"at Rostock\\
  Campus Ulmenstra\ss e 69\\
  D-18051 Rostock\\
  Germany}
\email{matthias.schymura@uni-rostock.de}
%\address{Institut f\"ur Mathematik\\
%  BTU Cottbus-Senftenberg\\
%  Platz der Deutschen Einheit 1\\
%  D-03046 Cottbus\\
%  Germany}
%\email{schymura@b-tu.de}

% \dedicatory{}

\begin{abstract}
Let $K \subseteq \RR^d$ be a convex body and let $\bw \in \inter(K)$ be an interior point of~$K$. The \emph{coefficient of asymmetry} $\ca(K,\bw) := \min\{ \lambda \geq 1 : \bw - K \subseteq \lambda (K - \bw) \}$ has been studied extensively in the realm of Hensley's conjecture on the maximal volume of a
$d$-dimensional lattice polytope that contains a fixed positive number of interior
lattice points. We study the coefficient of asymmetry for \emph{lattice zonotopes}, i.e.,
Minkowski sums of line segments with integer endpoints. Our main result gives the
existence of an interior lattice point whose coefficient of asymmetry is bounded above by
an explicit constant in $\Theta(d \log\log d)$, for any lattice zonotope that has an
interior lattice point. Our work is both inspired by and feeds on Wills' lonely runner
conjecture from Diophantine approximation: we make intensive use of a discrete version of
this conjecture, and reciprocally, we reformulate the lonely runner conjecture in terms
of the coefficient of asymmetry of a zonotope.
\end{abstract}

%\keywords{Lonely Runner Conjecture, lattice zonotopes, coefficient of asymmetry}

%\subjclass[2010]{Primary 52C07; Secondary 11J71, 52C17.}
% 11J71 Distribution modulo one
% 52C07 Lattices and convex bodies in $n$ dimensions
% 52C17 Packing and covering in $n$ dimensions

\date{28 January 2023}
 
\thanks{We thank Gennadiy Averkov for valuable discussions and the pointer to the
argument in the proof of Proposition~\ref{prop:ca-versus-minkowski-norm}, and J\"org Wills for pointing us to the references~\cite{schark1974eine,scharkwills1973asymptotisches}.}

\maketitle

% ####################
\section{Introduction}
% ####################

We study two seemingly disjoint problems, namely the determination of the coefficients of asymmetry of zonotopes -- a concept from the geometry of numbers -- and the lonely runner conjecture from Diophantine approximation.
As we will see, the two problems feed on each other in more than one way. We start by introducing them one at a time.

The following conjecture was raised by J\"org M.~Wills in the 1960's~\cite{willslonelyrunner}.

\begin{lrc}
Given pairwise distinct integers $n_0, n_1, \dots, n_d$, for each $0 \leq i \leq d$ there exists a real number~$t$ such that for all $0 \le j \le d$, $i \ne j$, the distance of $t \, ( n_i - n_j )$ to the nearest integer is at least $\frac{ 1 }{ d+1 }$.
\end{lrc}

Wills originally formulated this conjecture for \emph{real} numbers $n_0, n_1, \dots, n_d$, but it can be reduced to the rational and thus integral case~\cite{bohmanholzmankleitman2001six,henzemalikiosis}.
The lower bound $\frac{ 1 }{ d+1 }$ is best possible, as the case $n_j = j$ for $0 \le j
\le d$ and a classic result of Dirichlet on Diophantine approximation (see, e.g.,
\cite{cassels}) show.
The name \emph{Lonely Runner Conjecture}, introduced by Goddyn in~\cite{bieniagoddynetal}, stems from the charming model of $d+1$ runners going at different constant speeds around a
circular track of length~1 (having started at the same place and time); the conjecture says that each of them will at some point have distance at least $\frac{ 1 }{ d+1 }$ to the other runners.
For more on the Lonely Runner Conjecture's history, proofs for $d \le 6$, and connections to Diophantine approximation, view-obstruction problems, and graph theory, see
\cite{betkewills, % k <= 3
cusick1973viewobstruction,
cusickpomerance, % k=4
bohmanholzmankleitman2001six, % k=5, relaxation to rationals
barajasserra, % k=6
6lrcRenault, % k=6 simpler
tao2018someremarks}. % Tao recent 

A simple observation leads to a more convenient formulation of the problem:
The distance of any two runners at any given time depends only on their relative speeds.
So we may pick a fixed runner, say the one with speed $n_0$, reduce the speed of every runner by $n_0$ and consider only the loneliness of the first runner that is now stagnant.
So with
\begin{equation}\label{eqn:run-d}
  \run(d) := \inf_{\alpha \in \ZZ^d} \,\,\,\sup_{Q \in \RR} \,\,\,\min_{1 \leq i \leq d} \|
Q \alpha_i \|_\ZZ\, ,
\end{equation}
where $\| \cdot \|_\ZZ$ denotes the distance to the nearest integer, we may restate the
lonely runner conjecture as
\[
  \run(d) = \frac{ 1 }{ d+1 } \, .
\]

% \begin{lrc}
% Given pairwise distinct positive integers $n_1, n_2, \dots, n_d$, there exists a real number $t$ such that for all $1 \le j \le d$ the distance of $t \, n_j$ to the nearest integer is at least $\frac{ 1 }{ d+1 }$.
% \end{lrc}

A related quantity, stemming from an inhomogeneous Diophantine approximation problem (we
will give more details in Section~\ref{sect:lonelyrabbitsrunners} below), is
\begin{align}
\rab(d) := \inf_{\alpha \in (\RR \setminus \ZZ)^d} \,\,\,\sup_{Q \in \ZZ} \,\,\,\min_{1 \leq i \leq d} \| Q \alpha_i \|_\ZZ \, .\label{eqn:psi-d}
\end{align}
The lonely runner view is that now the runners are only allowed to take simultaneous \emph{jumps} rather than continuous moves.
(Also each of them has a non-integer velocity.)
We therefore % speak of \emph{rabbits} instead of runners, and 
refer to the challenge of determining~$\rab(d)$ as the \emph{Lonely Rabbit Problem}.
Contrary to the situation of the lonely runner, this problem has long been solved.
The first three values and a first general bound were provided by Wills~\cite{wills1968zursimultanenII}, who showed that
\begin{align}
\rab(1) = \frac13 \ ,\ \rab(2) = \frac15\ ,\ \rab(3) = \frac18 \quad \textrm{ and } \quad
\rab(d) \leq \frac{1}{6(d-2)} \ \textrm{ for } d \geq 4.\label{eqn:psi-d-wills-bound}
\end{align}
Cusick~\cite{cusick1972simultaneous} conjectured the precise value of $\rab(d)$, for any given~$d$, and proved his claim for $d \leq 7$.
Confirming Cusick's conjecture in general, Schark~\cite{schark1974eine} completely solved the Lonely Rabbit Problem by proving that
\begin{align}
\rab(d) &= \frac{1}{w(d)} \quad \textrm{ with } \quad w(d) := \max\left\{ z \in \NN : \tfrac12 \varphi(z) + h(z) \leq d\right\},\label{eqn:psi-d-w-d}
\end{align}
where
$
\varphi(z) = \#\left\{ a \in \NN : 1 \leq a \leq z, \ \gcd(a,z) = 1\right\}
$
is Euler's totient function and
\[
h(z) :=
\begin{cases}
0 &\textrm{ if } z \textrm{ is prime}, \\
h &\textrm{ if } z \textrm{ is composite},
\end{cases}
\]
with $h$ being the number of distinct prime divisors of a composite number~$z$.
On the asymptotic side, it was shown by Schark \& Wills~\cite{scharkwills1973asymptotisches} that
\begin{align}
\lim_{d \to \infty} \rab(d) \cdot d \cdot \log \log d = e^{-2 \gamma},\label{eqn:psi-d-asymptotics}
\end{align}
where $\gamma = 0.57721...$ is the Euler--Mascheroni constant.

Next we introduce the main geometric players of this paper.
Let $K \subseteq \RR^d$ be a convex body and let $\bw \in \inter(K)$ be an interior point of~$K$.
The \emph{coefficient of asymmetry} of~$\bw$ inside~$K$ is defined as
\begin{align}
\ca(K,\bw) &:= \max_{\br \in \RR^d \setminus \{\bzero\}} \frac{\max\{ \lambda > 0 : \bw + \lambda \br \in K \}}{\max\{ \lambda > 0 : \bw - \lambda \br \in K \}} = \min\left\{ \lambda \geq 1 : \bw - K \subseteq \lambda (K - \bw) \right\}.\label{eqn:ca-definition}
\end{align}
The first definition is worked with, for instance, in~\cite{pikhurko2001lattice} and the second in~\cite{gruenbaum1963measures}.
The equivalence of both definitions can be easily checked (see, e.g.,~\cite[Lem.~3.2.3]{toth2015measures}).
Clearly, $\ca(K,\bw) \geq 1$ and equality holds if and only if~$K$ is \emph{symmetric about~$\bw$}, meaning that $K - \bw = \bw - K$.
We collect a few more salient properties of $\ca(K,\bw)$ in Section~\ref{sect:coeffasymm}.

Given $\bn = (n_1,\ldots,n_d) \in \NN^d$ with distinct entries (which we may assume to
be relatively prime), let
\[
Z(\bn) := \RR \bn + [0,1]^d,
\]
which we call the \emph{lonely runner zonohedron} associated to $\bn$; here $+$ refers to
Minkowski (i.e., pointwise) sum.
Up to a translation and a dilation, the lonely runner zonohedron was introduced in~\cite{beckhostenschymura2018lrpolyhedra}, % as~$\cP(\bn)$.
following a view-obstruction formulation of the Lonely Runner Problem by
Cusick~\cite{cusick1973viewobstruction}. Here we develop this geometric picture further
by deriving (in Section~\ref{sect:lonelyrunners}) the following equivalent formulation.

\begin{lrc}
\label{lrc:coefficient-of-asymmetry}
Let $\bn \in \ZZ^d_{>0}$ have pairwise distinct entries.
Then there exists an interior lattice point $\bw \in \inter(Z(\bn)) \cap \ZZ^d$ such that $\ca(Z(\bn),\bw) \leq d$.
\end{lrc}

Given $\bg_1,\ldots,\bg_m \in \RR^d$, let
\[
 \Zon(\bg_1,\ldots,\bg_m) := \sum_{j=1}^m [\bzero,\bg_j] \subseteq \RR^d ,
\]
the \emph{zonotope generated by} $\bg_1,\ldots,\bg_m$; alternatively, we can think of
$\Zon(\bg_1,\ldots,\bg_m)$ as the projection of $[0,1]^m$ under the matrix with column
vectors $\bg_1,\ldots,\bg_m$.
We call $\Zon(\bg_1,\ldots,\bg_m)$ a \emph{lattice zonotope} if $\bg_1,\ldots,\bg_m \in \ZZ^d$.
Thus $Z(\bn)$ is an infinite version of the \emph{lonely runner zonotope}
$\Zon(\be_1,\ldots,\be_d,\bn)$, where $\be_i$ is the $i$th coordinate unit vector.

The above geometric connection to the Lonely Runner Conjecture suggests a study of the
coefficient of asymmetry of lattice zonotopes, and the following is our main result.

\begin{theorem}
\label{thm:main-ca-precise-bound-zonotopes}
Let $Z = \Zon(\bg_1,\ldots,\bg_m) \subseteq \RR^d$ be a lattice zonotope with $m \geq d$ generators and such that $\inter(Z) \cap \ZZ^d \neq \emptyset$.
Then there exists an interior lattice point $\bw \in \inter(Z) \cap \ZZ^d$ such that
\[
\ca(Z,\bw) \leq w(d) - 1 \in \Theta(d \log\log d).
\]
Moreover, the bound is best possible.
\end{theorem}

We note that our above reformulation of the Lonely Runner Conjecture gives a strong
connection to the coefficient of asymmetry of an (infinite version of a) zonotope, and
reciprocally, Theorem~\ref{thm:main-ca-precise-bound-zonotopes} yields an intimate
connection of the coefficient of asymmetry of a zonotope with the lonely rabbit constant~$w(d) = \rab(d)^{-1}$.

The coefficient of asymmetry of interior lattice points of general lattice polytopes has been studied extensively in the realm of Hensley's conjecture on the maximal volume of a $d$-dimensional lattice polytope that contains a fixed positive number of interior lattice points.
We refer to~\cite{pikhurko2001lattice,averkovkruempelmannnill2020latticesimplices} and the references therein for more information.
The best-known bound to date is the following:

\begin{theorem}[{Averkov, Kr\"umpelmann \& Nill~\cite[Proof of Theorem~1.4]{averkovkruempelmannnill2020latticesimplices}}]
\label{thm:ca-general-lattice-polytopes}
Let $P \subseteq \RR^d$ be a lattice polytope with $\inter(P) \cap \ZZ^d \neq \emptyset$.
Then there is an interior lattice point $\bw \in \inter(P) \cap \ZZ^d$ with
\[
\ca(P,\bw) \leq d(2d+1)\left( s_{2d+1} -1 \right) - 1,
\]
where $s_1 = 2$ and $s_i = 1 + s_1 \cdot\dots\cdot s_{i-1}$, for $i \geq 2$, denotes the double exponentially growing Sylvester sequence.
\end{theorem}

The optimal such bound needs to be of double exponential growth in dependence on the dimension~$d$, as the so-called Zaks-Perles-Wills-simplices show (see~\cite{averkovkruempelmannnill2020latticesimplices}).

Our paper is organized as follows.
In Section~\ref{sect:coeffasymm} we review fundamental properties of the coefficient of asymmetry.
Section~\ref{sect:lonelyrunners} discusses the derivation of the equivalent formulation of the Lonely Runner Conjecture in terms of deep interior lattice points in the zonohedron~$Z(\bn)$.
Afterwards in Section~\ref{sect:lonelyrabbitsrunners} we describe the Diophantine approximation problem behind the lonely rabbits problem in detail, and discuss examples of non-integral vectors~$\alpha$ that attain~$\rab(d)$.
With these preparations we give the proof of Theorem~\ref{thm:main-ca-precise-bound-zonotopes} in two steps.
First, we establish it for lattice parallelepipeds in Section~\ref{sect:deep-parallelepipeds} and construct examples that attain the stated upper bound $w(d)-1$.
Second, in Section~\ref{sect:deeplattptszonotopes}, we derive a Carath\'{e}odory/Steinitz type result for zonotopes which enables us to reduce Theorem~\ref{thm:main-ca-precise-bound-zonotopes} to parallelepipeds.
Complementing the discussion of lattice zonotopes that relate to the lonely runner problem, we devote Section~\ref{sect:ehrhart} to the determination of the number of interior lattice points of the lonely runner zonotope and derive a geometric intuition for a finite checking result of Tao~\cite{tao2018someremarks} for the lonely runner conjecture.
We conclude the paper with an outlook for future work and some open questions in Section~\ref{sect:outlook}.

% ##################################################
\section{The Coefficient of Asymmetry}
\label{sect:coeffasymm}
% ##################################################

We start by collecting basic properties of the coefficient of asymmetry which we will continuously  use in the remainder of the paper.
Most of the statements below are well-known, folklore or follow from first principles.
Since they are rather scattered in the literature~\cite{pikhurko2001lattice,gruenbaum1963measures} and appear often without proof, we aim to give a complete account.

Our first observation is that the coefficient of asymmetry cannot increase under linear transformations.

\begin{proposition}
\label{prop:ca-projections}
Let $K \subseteq \RR^d$ and $L \subseteq \RR^m$ be convex bodies such that there is a linear map $A : \RR^m \to \RR^d$ with $K = AL$.
Then for every $\bw \in \inter(K)$ and every $\bv \in A^{-1}(\bw) \cap \inter(L)$,
\[ \ca(K,\bw) \leq \ca(L,\bv) \, . \]
Moreover, if $m=d$ and $A : \RR^d \to \RR^d$ is invertible, then $\ca(K,\bw) = \ca(L,\bv)$.
\end{proposition}

\begin{proof}
We use the second definition of $\ca(L,\bv)$ in~\eqref{eqn:ca-definition} and obtain 
\begin{align*}
\ca(L,\bv) &= \min\left\{ \lambda \geq 1 : \bv - L \subseteq \lambda (L - \bv)\right\} \ \geq \ \min\left\{ \lambda \geq 1 : A(\bv - L) \subseteq \lambda\, A(L - \bv)\right\} = \ca(K,\bw).
\end{align*}
For the second claim, we apply the just established inequality to both~$A$ and~$A^{-1}$.
\end{proof}

Now, let $K \subseteq \RR^d$ be a \emph{centrally symmetric} convex body, that is, $K$ is symmetric about~$\bzero$ so that $K = -K$.
Each such~$K$ induces a norm $\|\cdot\|_K : \RR^d \to \RR$ via
\[
\| \bx \|_K := \min\{\lambda \geq 0 : \bx \in \lambda K\}, \quad \textrm{ for } \quad \bx \in \RR^d.
\]
This norm function obeys an exact relationship with the coefficient of asymmetry:

\begin{proposition}
\label{prop:ca-versus-minkowski-norm}
Let $K \subseteq \RR^d$ be a centrally symmetric convex body and let $\bw \in \inter(K)$ be an interior point.
Then
\[
\ca(K,\bw) = \frac{1 + \|\bw\|_K}{1 - \|\bw\|_K} \qquad \textrm{ and } \qquad \|\bw\|_K =
\frac{\ca(K,\bw) - 1}{\ca(K,\bw) + 1} \, .
\]
\end{proposition}

\begin{proof}
The two relations are equivalent, so that it suffices to prove the second.
Moreover, we use the second definition in~\eqref{eqn:ca-definition} and recall that $\ca(K,\bw)$ equals the minimal $\lambda \geq 1$ such that $\bw - K \subseteq \lambda (K - \bw)$.
In view of the symmetry $K = -K$ and the cancellation laws for Minkowski addition of convex bodies (see,
e.g.,~\cite[p.~48]{schneider2014convexbodies}), the following equivalences hold for every~$\lambda \geq 1$:
\begin{align*}
\bw - K \subseteq \lambda (K - \bw) & \quad\Longleftrightarrow\quad (1 + \lambda) \bw + K \subseteq \lambda K \quad\Longleftrightarrow\quad (1 + \lambda) \bw \in (\lambda - 1 )K.
\end{align*}
Therefore, $\ca(K,\bw) \leq \lambda$ if and only if $\|\bw\|_K \leq \frac{\lambda - 1}{\lambda + 1}$, which implies $\|\bw\|_K = \frac{\ca(K,\bw) - 1}{\ca(K,\bw) + 1}$.
\end{proof}

As a useful corollary, we obtain that the coefficient of asymmetry is monotonically non-increasing on centrally symmetric convex bodies:

\begin{corollary}
\label{cor:ca-monotonicity}
Let $K,L \subseteq \RR^d$ be centrally symmetric convex bodies with $K \subseteq L$, and let $\bw \in \inter(K)$ be an interior point in~$K$.
Then $\ca(L,\bw) \leq \ca(K,\bw)$.
\end{corollary}

\begin{proof}
By definition of the norm associated with~$K$ and~$L$, we have $\| \bw \|_L \leq \| \bw \|_K$.
In view of Proposition~\ref{prop:ca-versus-minkowski-norm} this implies
\[
\ca(L,\bw) = \frac{1 + \| \bw \|_L}{1 - \| \bw \|_L} \leq \frac{1 + \| \bw \|_K}{1 - \| \bw \|_K} = \ca(K,\bw),
\]
where we have also used that the function $x \mapsto \frac{1+x}{1-x}$ is non-decreasing on~$[0,1)$.
\end{proof}

Note that the symmetry condition on~$K$ and~$L$ is crucial.
For example, consider $K = [-1,1]^d$, $\bv \in \RR^d \setminus K$ and let $L$ be the convex hull of~$K$ and~$\bv$.
Then for the center $\bw = \bzero$ of~$K$, we have $\ca(K,\bw) = 1 < \ca(L,\bw)$, since otherwise~$L$ would be centrally symmetric as well.

Furthermore, $\ca(L,\bv)$ is invariant under simultaneous translations and scalings of $L$ and~$\bv$.
More precisely:

\begin{proposition}
\label{prop:ca-covariance-translation-scaling}
Let $K \subseteq \RR^d$ be a convex body with an interior point $\bw \in \inter(K$).
Then for every translation vector $\bt \in \RR^d$ and every scaling factor $s > 0$, 
\[
\ca(K + \bt,\bw + \bt) = \ca(K,\bw) \quad\textrm{ and } \quad \ca(s \, K,s \, \bw) = \ca(K,\bw).
\]
\end{proposition}

\begin{proof}
This follows directly from the second definition in~\eqref{eqn:ca-definition}.
\end{proof}

In later sections, we often need the precise description of the coefficient of asymmetry of an interior point in a lattice parallelepiped, and thus record this as a corollary of the properties above:

\begin{corollary}
\label{cor:ca-parallelepipeds}
Let $P = [\bzero,\bv_1] + \dots + [\bzero,\bv_d] \subseteq \RR^d$ be a lattice parallelepiped and let $\bw = \sum_{i=1}^d \alpha_i \bv_i \in \inter(P)$, for some $\alpha = (\alpha_1,\ldots,\alpha_d) \in (0,1)^d$.
Then
\[
\ca(P,\bw) = \frac{\frac12 + \max_{1 \leq i \leq d} |\alpha_i - \frac12|}{\frac12 - \max_{1 \leq i \leq
d} |\alpha_i - \frac12|} \, .
\]
%\begin{enumerate}[(i)]
% \item The coefficient of asymmetry of~$\bw$ inside~$P$ is attained in direction of an edge~$[\bzero,\bv_i]$ of~$P$.
% \item\label{prop:ca-parallelepipeds:2} We have
% \[
% \ca(P,\bw) = \max_{1 \leq i \leq d} \frac{\frac12 + |\alpha_i - \frac12|}{\frac12 - |\alpha_i - \frac12|} = \frac{\frac12 + \max_{1 \leq i \leq d} |\alpha_i - \frac12|}{\frac12 - \max_{1 \leq i \leq d} |\alpha_i - \frac12|}.
% \]
%\end{enumerate}
\end{corollary}

\begin{proof}
Using Propositions~\ref{prop:ca-projections},~\ref{prop:ca-covariance-translation-scaling}, and~\ref{prop:ca-versus-minkowski-norm} (in this order), we obtain 
\[
\ca(P,\bw) = \ca \left([0,1]^d,\alpha\right) = \ca\left([-1,1]^d,2\alpha - \bone\right) = \frac{1 + \|2\alpha -
\bone\|_\infty}{1 - \|2\alpha - \bone\|_\infty} = \frac{\frac12 + \max_{1 \leq i \leq d} |\alpha_i -
\frac12|}{\frac12 - \max_{1 \leq i \leq d} |\alpha_i - \frac12|} \, ,
\]
where $\bone=(1,\dots,1)^\intercal$ denotes the all-one vector and $\| \cdot \|_\infty$ the maximum norm.
\end{proof}

The second definition in~\eqref{eqn:ca-definition} of the coefficient of asymmetry can be extended to any, possibly unbounded, closed convex set $C \subseteq \RR^d$ with respect to an interior point $\bw \in \inter(C)$, via
\[
\ca(C,\bw) = \min\left\{ \lambda \geq 1 : \bw - C \subseteq \lambda (C - \bw) \right\} \, .
\]
If the recession cone $\rec(C) := \left\{ \br \in \RR^d : \bx + \br \in C \text{ for all }
\bx \in C \right\}$ of~$C$ is a linear subspace $L \subseteq \RR^d$, then we
obtain the following simple description of the coefficient of asymmetry of~$C$ by projecting out~$L$.

\begin{proposition}
\label{prop:ca-linear-recession-cone}
Let $C \subseteq \RR^d$ be a closed convex set such that $L = \rec(C)$ is a linear subspace.
Then, for every $\bw \in \inter(C)$, 
\[
\ca(C,\bw) = \ca \left( C | L^\perp, \bw | L^\perp \right) .
\]
\end{proposition}

\begin{proof}
For the sake of brevity, we write $C_L = C | L^\perp$ and $\bw_L = \bw | L^\perp$.
Since $L$ is a subspace, $C = C_L + L$ and $\bw = \bw_L + \by$, for some $\by \in L$.
Thus
\begin{align*}
\ca(C,\bw) &= \min\left\{ \lambda \geq 1 : (\bw_L + \by) - (C_L + L) \subseteq \lambda \left((C_L + L) - (\bw_L + \by) \right) \right\} \\
&= \min\left\{ \lambda \geq 1 : \bw_L - C_L \subseteq \lambda (C_L - \bw_L) \right\} = \ca(C | L^\perp, \bw | L^\perp) \, .\qedhere
\end{align*}
\end{proof}

% #######################################################
\section{Lonely Runners and Lonely Runner Zonohedra}
\label{sect:lonelyrunners}
% #######################################################

Let $\bn \in \ZZ_{>0}^d$ be a \emph{velocity vector}, by which we mean its $d \geq 2$ entries are distinct and relatively prime.
For such an~$\bn$, we denote by
\[
\gap(\bn) := \max_{\beta \in \RR} \,\,\,\min_{1 \leq i \leq d} \|\beta n_i\|_\ZZ = \max_{\beta \in [0,1]} \,\,\,\min_{1 \leq i \leq d} \|\beta n_i\|_\ZZ
\]
the \emph{gap of loneliness} of~$\bn$.
We can restrict to the range $\beta \in [0,1]$ since the $n_i$ are all integers, and we
can define $\gap(\bn)$ by a maximum because it is attained at some $\beta = \frac{a}{n_i +
n_j}$, for some $i \neq j$ and some integer $0 < a < n_i + n_j$ (see~\cite[Theorem~6]{czerwinskigrytczuk2008invisible}).
Observe that we always have $\gap(\bn) > 0$ by using an irrational multiple~$\beta$.
The original version of Wills' conjecture is as follows. 

\begin{lrc}
For any velocity vector $\bn \in \ZZ^d_{>0}$ we have $\gap(\bn) \ge \frac{ 1 }{ d+1 }$.  
\end{lrc}

The well-known (and easily seen to be equivalent) visibility version is:

\begin{lrc}
For any velocity vector $\bn \in \ZZ^d_{>0}$, there exists $\bm \in \ZZ^d$ such that  
\begin{equation}\label{eq:visibilitylrc}
%  \bn \in \RR_{ \ge 0 } \left( \bm + \tfrac 1 {d+1} \bone + \tfrac{ d-1 }{ d+1 } [0,1]^d \right).
  \bm \in \RR \bn + \left[\tfrac{ 1 }{ d+1 },\tfrac{ d }{ d+1 }\right]^d \,.
\end{equation}
\end{lrc}
We seek to understand the gap of loneliness geometrically in terms of
the lonely runner zonohedron.
This will allow us to reformulate the Lonely Runner Conjecture via the coefficient of asymmetry.
To this end, we note that \eqref{eq:visibilitylrc} is equivalent to
%\[
%  \bm \in \RR \bn - \tfrac 1 {d+1} \bone - \tfrac{ d-1 }{ d+1 } [0,1]^d ,
%\]
%which in turn (by symmetry) is equivalent to
\[
  \bm \in \RR \bn + \tfrac 1 {d+1} \bone + \tfrac{ d-1 }{ d+1 } [0,1]^d = \tfrac 1 {d+1} \bone +
\tfrac{ d-1 }{ d+1 } \, Z(\bn) \, ,
\]
which yields the reformulation
\begin{lrc}
For any velocity vector $\bn \in \ZZ^d_{>0}$ the zonohedron $\tfrac 1 {d+1} \bone + \tfrac{
d-1 }{ d+1 } \, Z(\bn)$ contains a lattice point.
\end{lrc}

% A set $C \subseteq \RR^d$ is centrally symmetric about the point $\bc \in \RR^d$ if $C - \bc = -(C - \bc)$.
%By definition, the lonely runner zonotope~$\Zon(\be_1,\ldots,\be_d,\bn)$ is centrally symmetric about $\bc_{\bn} := \frac12(\bone + \bn)$, whereas 
The lonely runner zonohedron~$Z(\bn)$ is symmetric about any point on the line~$\frac12 \bone + \RR \bn$, in particular about $\bc_{\bn} := \frac12(\bone + \bn)$.
The following characterization allows for a geometric interpretation of the gap of loneliness $\gap(\bn)$ in terms of~$Z(\bn)$.

\begin{proposition}
\label{prop:lrc-gap-loneliness-interpretation}
Let $\bn \in \ZZ^d_{>0}$ be a velocity vector and let $0 < \gamma \leq \frac12$.
The following statements are equivalent:
\begin{enumerate}[(i)]
 \item There is some $\beta \in \RR$ such that $\|\beta n_i \|_\ZZ \geq \gamma$, for every $1 \leq i \leq d$.
 \item There is some $\bw \in \inter(Z(\bn)) \cap \ZZ^d$ such that $\bw \in \bc_\bn + (1 - 2 \gamma) \left( Z(\bn) - \bc_\bn \right)$.
\end{enumerate}
\end{proposition}

\begin{proof}
We first prove the implication $(i) \Longrightarrow (ii)$:
If $\gamma = \frac12$, then $\| \beta n_i \|_\ZZ = \frac12$, which means that $n_i$ is odd, for every $1 \leq i \leq d$.
Consequently, $\bw = \bc_\bn = \frac12 (\bone + \bn)$ is a lattice point satisfying~(ii).
Now let $\gamma < \frac12$.
By assumption, there are integers $z_i \in \ZZ$ such that $z_i + \gamma \leq \beta n_i \leq z_i + 1 - \gamma$, for $1 \leq i \leq d$.
Hence, there are $\alpha_1,\dots,\alpha_d \in [0,1]$ such that $\beta n_i = z_i + \gamma + (1-2\gamma) \alpha_i$, for $1 \leq i \leq d$.
Since $\gamma + (1-2\gamma) \alpha_i \in [\gamma, 1 - \gamma]$, we write $\bw =
-(z_1,\dots,z_d) \in \ZZ^d$ and obtain 
\begin{align*}
\bw - \bc_\bn \in [\gamma,1-\gamma]^d - \beta \bn - \bc_\bn &= (1 - 2 \gamma) \left( [0,1]^d - \tfrac12 \bone \right) - (\beta + \tfrac12) \bn \\
&= (1 - 2 \gamma) \left( [0,1]^d + \left(\tfrac12 - \tfrac{\beta + \tfrac12}{1 - 2 \gamma}\right) \bn - \bc_\bn \right) \subseteq (1 - 2 \gamma) \left( Z(\bn) - \bc_\bn \right) \, .
\end{align*}
Because $\gamma > 0$, we have $1-2\gamma < 1$, and thus $\bw \in \bc_\bn + (1 - 2 \gamma) \left( Z(\bn) - \bc_\bn \right) \subseteq \inter(Z(\bn))$ as claimed.

Now, we prove $(ii) \Longrightarrow (i)$:
Let $\bw \in \ZZ^d$ with $\bw \in \bc_{\bn} + (1 - 2 \gamma) \left( Z(\bn) - \bc_{\bn} \right) = 2 \gamma \, \bc_{\bn} + (1 - 2 \gamma) Z(\bn)$.
Hence, there exist $\alpha_0 \in \RR$ and $\alpha_1,\dots,\alpha_d \in [0,1]$ such that
\[
\bw = \gamma (\bone + \bn) + (1 - 2 \gamma) \left( \alpha_0 \bn + \sum_{i=1}^d \alpha_i \be_i \right)
\]
and thus
\[
  \quad w_i = \gamma + (1 - 2 \gamma) \alpha_i + (\gamma + (1 - 2 \gamma) \alpha_0) n_i
\qquad \text{ for all } i \in [d].
\]
Writing $\beta = \gamma + (1 - 2 \gamma) \alpha_0$ and noting that $\gamma + (1 - 2
\gamma) \alpha_i \in [\gamma,1-\gamma]$, we obtain $\| \beta n_i \|_{\ZZ} \geq \gamma$, for every $i \in [d]$, because~$w_i \in \ZZ$.
\end{proof}

Writing
\[
\lambda(C,\bw) := \min\left\{ \lambda \geq 0 : \bw - \bc \in \lambda (C - \bc) \right\} \, ,
\]
for any closed convex set $C \subseteq \RR^d$ that is symmetric about~$\bc \in \RR^d$, Proposition~\ref{prop:lrc-gap-loneliness-interpretation} implies that
\begin{align}
\min_{\bw \in \inter(Z(\bn)) \cap \ZZ^d} \lambda(Z(\bn),\bw) = 1 - 2 \gap(\bn) \,.\label{eqn:min-lambda-gap}
\end{align}
Now, let $L = \rec(Z(\bn)) = \RR \bn$.
Then, for every $\bw \in \inter(Z(\bn))$, we have $\ca(Z(\bn),\bw) = \ca(Z(\bn)|L^\perp,\bw|L^\perp)$ in view of Proposition~\ref{prop:ca-linear-recession-cone}, and moreover for every $0 < \gamma \leq \frac12$ 
\begin{align}
\bw - \bc_\bn \in  (1 - 2 \gamma) \left( Z(\bn) - \bc_\bn \right) \qquad \Longleftrightarrow \qquad \bw|L^\perp - \bc_\bn|L^\perp \in (1 - 2 \gamma) \left( Z(\bn)|L^\perp - \bc_\bn|L^\perp \right) \,.\label{eqn:zonohedron-projection-scalar}
\end{align}
Because the projection $Z(\bn)|L^\perp$ is a zonotope that is symmetric about $\bc_\bn|L^\perp$, the right hand side of~\eqref{eqn:zonohedron-projection-scalar} can be understood as a bound on the length of the vector $\bw|L^\perp - \bc_\bn|L^\perp$ measured by the norm that is induced by $Z(\bn)|L^\perp - \bc_\bn|L^\perp$.
By virtue of Proposition~\ref{prop:ca-versus-minkowski-norm} and~\eqref{eqn:min-lambda-gap} this translates into
\begin{align}
\min_{\bw \in \inter(Z(\bn)) \cap \ZZ^d} \ca(Z(\bn),\bw) = \frac{1+(1 - 2 \gap(\bn))}{1-(1 - 2 \gap(\bn))} = \frac{1}{\gap(\bn)} - 1 \,.\label{eqn:ca-gap-relation}
\end{align}
Since the Lonely Runner Conjecture states that $\gap(\bn) \geq \frac{1}{d+1}$ for every velocity vector $\bn \in \ZZ_{>0}^d$, we established the desired reformulation in terms of the coefficient of asymmetry.

\begin{lrc}
Let $\bn \in \ZZ^d_{>0}$ be a velocity vector.
Then there exists an interior lattice point $\bw \in \inter(Z(\bn)) \cap \ZZ^d$ such that $\ca(Z(\bn),\bw) \leq d$.
\end{lrc}

We conclude with a localization version of this reformulation.
A natural choice for a Lonely Runner \emph{Zonotope} would be $Z_\bn := \Zon(\be_1,\dots,\be_d,\bn)$, which has center~$\bc_\bn$ (we study this zonotope in more detail in Section~\ref{sect:ehrhart}).
But in general, we cannot just replace the zonohedron~$Z(\bn)$ above by~$Z_\bn$ and obtain an equivalent version, because the integral translates $k \bn + \bc_\bn + \tfrac{d-1}{d+1}(Z_\bn - \bc_\bn)$, $k \in \ZZ$, may not cover~$\bc_\bn + \tfrac{d-1}{d+1}(Z(\bn) - \bc_\bn)$, so that we may miss lattice points that are present in the formulation with respect to~$Z(\bn)$.
One can check that this happens, for example, for the velocity vector $\bn = (2,5)^\intercal$.
% \matthias{I miscalculated the examples $(1,1)$ and $(1,2)$ from an earlier email, but it works out badly with $(2,5)$..}

However, taking a suitable dilate of~$\bn$ as the last generator will do.
For instance, for $T_\bn := \Zon(\be_1,\ldots,\be_d,3\bn)$ we have
\[
  \bc_\bn + \tfrac{ d-1 }{ d+1 } ( Z(\bn) - \bc_\bn ) = \bigcup_{ k \in \ZZ } \left( k \bn + \bc_\bn + \tfrac{ d-1 }{ d+1
} ( T_\bn - \bc_\bn ) \right) \,.
\]
% for any $\bx \in \tfrac{ d-1 }{ d+1 } \, T_\bn$ there exists $\mu \in [0,1]$ such that $[\bx - \mu \bn, \bx + (1-\mu) \bn \subset \tfrac{ d-1 }{ d+1 } \, T_\bn$
Thus the Lonely Runner Conjecture is equivalent to the zonotope 
$\tfrac 1 {d+1} \bone + \tfrac{ d-1 }{ d+1 } \, T_\bn$ containing a lattice point, for 
any velocity vector $\bn \in \ZZ^d_{>0}$. In terms of the coefficient of asymmetry, this
reads as follows. 

\begin{lrc}
Let $\bn \in \ZZ^d_{>0}$ be a velocity vector.
Then there exists an interior lattice point $\bw \in \inter(T_\bn) \cap \ZZ^d$ such that $\ca(T_\bn,\bw) \leq d$.
\end{lrc}

% #######################################################
\section{Simultaneous Diophantine Approximation Problems and Lonely Rabbits}
% #######################################################
\label{sect:lonelyrabbitsrunners}

A classical result in simultaneous Diophantine approximation theory is \emph{Dirichlet's Approximation
Theorem} (see, e.g.,~\cite[\S11.12]{hardywright1979anintroduction}).
It states that for any real numbers $\alpha_1,\ldots,\alpha_d \in \RR$ there are integers $Q,P_1,\ldots,P_d \in \ZZ$ with $Q > 0$ such that
\[
|Q \alpha_i - P_i| \leq \frac{1}{Q^{\frac{1}{d}}},\quad\textrm{ for all }\quad 1 \leq i \leq d.
\]
Motivated by its utility for our proof of Theorem~\ref{thm:main-ca-precise-bound-zonotopes}, we are interested in an \emph{inhomogeneous} variant of Dirichlet's theorem.
More precisely, we define $\delta_d > 0$ as the smallest positive number such that for any $\alpha_1,\ldots,\alpha_d \in \RR \setminus \ZZ$, there are integers $Q,P_1,\ldots,P_d \in \ZZ$ with
\begin{align}
\left|Q \alpha_i - P_i - \frac12\right| \leq \delta_d,\quad\textrm{ for all }\quad 1 \leq i \leq d.\label{eqn:inhomDA}
\end{align}
Since we assume that all the $\alpha_i$ are non-integral, we clearly have $\delta_d < \frac12$, but of course we are interested in a much better bound, or even the exact value of $\delta_d$.
For every~$d$, we have $\delta_d \leq \delta_{d+1}$, which can be seen by just repeating one of the $\alpha_j$ from an extremal set of numbers attaining~$\delta_d$.

It turns out that the inhomogeneous Diophantine approximation problem of determining $\delta_d$ can be phrased as a discrete variant of the Lonely Runner Conjecture.
To see this, write $\rab(d) := \frac12 - \delta_d$.
Then the condition
\[
\exists \, Q,P_1,\ldots,P_d \in \ZZ \quad \textrm{ such that } \quad \max_{1 \leq i \leq d} \left|Q \alpha_i - P_i - \frac12\right| \leq \delta_d = \frac12 - \rab(d)
\]
is equivalent to saying that
\[
\exists \, Q \in \ZZ \quad \textrm{ such that } \quad \|Q\alpha_i\|_\ZZ \geq \rab(d) \quad \textrm{ for every } \quad 1 \leq i \leq d \,.
\]
Thus, the constant $\rab(d)$ can be compactly written as~\eqref{eqn:psi-d}.
The first $20$ values of $w(d)$, $\rab(d)$ and~$\delta_d$ are written down in Table~\ref{tbl:psi-d-values}.
\begin{table}[ht]
\centering
\captionsetup{position=below}
\begin{tabular}{|c||c|c|c|c|c|c|c|c|c|c|c|c|c|c|c|c|c|c|c|c|}
\hline $d$      & $1$           & $2$           & $3$           & $4$            & $5$            & $6$            & $7$            & $8$            & $9$            & $10$           & $11$           & $12$           & $13$ & $14$ & $15$ & $16$ & $17$ & $18$ & $19$ & $20$           \\ \hline
$w(d)$ & $3$ & $5$ & $8$ & $12$ & $18$ & $24$ & $30$ & $36$ & $42$ & $48$ & $60$ & $60$ & $66$ & $72$ & $90$ & $90$ & $90$ & $96$ & $120$ & $120$ \\ \hline
$\rab(d)$ & $\frac{1}{3}$ & $\frac{1}{5}$ & $\frac{1}{8}$ & $\frac{1}{12}$ & $\frac{1}{18}$ & $\frac{1}{24}$ & $\frac{1}{30}$ & $\frac{1}{36}$ & $\frac{1}{42}$ & $\frac{1}{48}$ & $\frac{1}{60}$ & $\frac{1}{60}$ & $\frac{1}{66}$ & $\frac{1}{72}$ & $\frac{1}{90}$ & $\frac{1}{90}$ & $\frac{1}{90}$ & $\frac{1}{96}$ & $\frac{1}{120}$ & $\frac{1}{120}$ \\ \hline
$\delta_d$ & $\frac{1}{6}$ & $\frac{3}{10}$ & $\frac{3}{8}$ & $\frac{5}{12}$ & $\frac{4}{9}$ & $\frac{11}{24}$ & $\frac{7}{15}$ & $\frac{17}{36}$ & $\frac{10}{21}$ & $\frac{23}{48}$ & $\frac{29}{60}$ & $\frac{29}{60}$ & $\frac{16}{33}$ & $\frac{35}{72}$ & $\frac{22}{45}$ & $\frac{22}{45}$ & $\frac{22}{45}$ & $\frac{47}{96}$ & $\frac{59}{120}$ & $\frac{59}{120}$ \\ \hline
\end{tabular}
\caption{The first $20$ values of $w(d)$ and of the approximation constants $\rab(d)$ and $\delta_d$.}
\label{tbl:psi-d-values}
\end{table}

The definition~\eqref{eqn:psi-d-w-d} of the parameter $w(d)$ suggests what an extremal set of numbers $0 < \alpha_1,\ldots,\alpha_d < 1$ attaining~$\rab(d)$ looks like.
To this end, given a vector $\alpha = (\alpha_1,\ldots,\alpha_d) \in (\RR \setminus \ZZ)^d$ of non-integers, we write
\[
\psi(\alpha) := \sup_{Q \in \ZZ} \,\,\,\min_{1 \leq i \leq d} \| Q \alpha_i \|_\ZZ \, .
\]
%Note that for any integral vector $\bx \in \ZZ^d$, we have $\psi(\alpha + \bx) = \psi(\alpha)$.
Furthermore, for $z \in \NN$, we write
\[
\Phi(z) := \left\{ a \in \ZZ : 1 \leq a \leq \frac{z}{2}, \ \gcd(a,z) = 1 \right\} .
\]
% for the first half of integers coprime to~$z$.
%With this notation we can now describe vectors $\alpha = (\alpha_1,\ldots,\alpha_d)$ that attain~$w(d)$ and thus~$\rab(d)$.
Extremal vectors have been identified by Cusick~\cite[Proof of Lemma~1]{cusick1972simultaneous}.
We give the arguments here for completeness and because, later in Section~\ref{sect:deep-parallelepipeds}, we turn those vectors into lattice zonotopes that attain the bound in Theorem~\ref{thm:main-ca-precise-bound-zonotopes}.

\begin{proposition}
\label{prop:psi-d-extremal}
Let the prime factorization of an integer $z \in \NN$ be given by $z = \prod_{i=1}^h p_i^{c_i}$, where $p_1 < p_2 < \dots < p_h$ and $c_i \geq 1$.
Further, write $k = \frac12 \phi(z)$ and let $a_1, a_2, \ldots, a_k$ be the elements of $\Phi(z)$ labeled in increasing order.
We define $\alpha^z \in (\QQ \setminus \ZZ)^{k + h(z)}$ by
\[
\alpha^z_i = \frac{a_i}{z}\quad \textrm{ for }\quad 1 \leq i \leq k \quad \textrm{ and } \quad \alpha^z_{k+j} = \frac{1}{p_j} \quad \textrm{ for } 1 \leq j \leq h(z).
\]
Then $\psi(\alpha^z) = \frac{1}{z}$.
\end{proposition}

\begin{proof}
First, since $\frac{1}{z} \leq \alpha^z_i \leq \frac{z-1}{z}$, for every $1 \leq i \leq k+h(z)$, we clearly have $\psi(\alpha^z) \geq \frac{1}{z}$.
For the reverse inequality, we need to show that for every integer $Q \in \ZZ$ there is some $1 \leq i \leq k+h(z)$, such that $\|Q \alpha^z_i\|_\ZZ \leq \frac{1}{z}$.
This is certainly true for all~$Q$ with $\gcd(Q,z) > 1$, because of $\alpha^z_{k+j}=\frac{1}{p_j}$, for $1 \leq j \leq h(z)$.

Thus, let $Q \in \ZZ$ be such that $\gcd(Q,z) = 1$.
Then $Q$ has a multiplicative inverse $P \in \{1,\ldots,z-1\}$ modulo~$z$, that is, $Q P \equiv 1 \!\!\mod\,z$.
If $1 \leq P \leq \frac{z}{2}$, then $P \in \Phi(z)$.
Thus, $\frac{P}{z}$ equals one of the first~$k$ entries of~$\alpha$, and satisfies $\|Q \frac{P}{z}\|_\ZZ = \frac{1}{z}$.
If $\frac{z}{2} < P < z$, then $P' = z - P \in \Phi(z)$ and thus $\frac{P'}{z}$ equals one of the first~$k$ entries of~$\alpha$, and satisfies $\|Q \frac{P'}{z}\|_\ZZ = \|\frac{-1}{z}\|_\ZZ = \frac{1}{z}$.
\end{proof}

Now, if we take $z = w(d)$, then the vector $\alpha^z$ from Proposition~\ref{prop:psi-d-extremal} has dimension $\frac12 \varphi(z) + h(z) \leq d$.
In the case that $m = d - \frac12 \varphi(z) - h(z) \geq 1$, we repeat the first coordinate $\frac{1}{z}$ of $\alpha^z$ $m$ times and get a vector $\bar\alpha^z \in (\QQ \setminus \ZZ)^d$ with $\psi(\bar\alpha^z) = \frac{1}{w(d)} = \rab(d)$.

% ######################################################
\section{Deep Lattice Points in Lattice Parallelepipeds}
\label{sect:deep-parallelepipeds}
% ######################################################

In this section, we aim to prove Theorem~\ref{thm:main-ca-precise-bound-zonotopes} for the case of lattice parallelepipeds.
We use the simultaneous Diophantine approximation problem behind the constant~$\delta_d$
from Section~\ref{sect:lonelyrabbitsrunners} in order to adjust the approach of
Pikhurko~\cite{pikhurko2001lattice} which established the existence of deep interior lattice points in lattice \emph{simplices}.

\begin{theorem}
\label{thm:ca-parallelepipeds}
Let $P \subseteq \RR^d$ be a lattice parallelepiped such that $\inter(P) \cap \ZZ^d \neq \emptyset$.
Then there exists an interior lattice point $\bw \in \inter(P) \cap \ZZ^d$ such that
\[
\ca(P,\bw) \leq \frac{\frac12 + \delta_d}{\frac12 - \delta_d} = w(d) - 1.
\]
\end{theorem}

\begin{proof}
Write $P = [\bzero,\bv_1] + \dots + [\bzero,\bv_d] \subseteq \RR^d$ for suitable
generators $\bv_1,\ldots,\bv_d \in \ZZ^d$, and let $\bw = \sum_{i=1}^d \alpha_i \bv_i \in \inter(P) \cap \ZZ^d$ be an interior lattice point in~$P$.
In particular this means that $0 < \alpha_1,\ldots,\alpha_d < 1$.
In view of Corollary~\ref{cor:ca-parallelepipeds}, the coefficient of asymmetry $\ca(P,\bw)$ corresponds to $\max_{1 \leq i \leq d}|\alpha_i - \frac12|$.
If this maximum is not yet itself bounded above by~$\delta_d$, we ``jump'' inside~$P$ with the help of the inhomogeneous problem in~\eqref{eqn:inhomDA}.
More precisely, we find integers $Q,P_1,\ldots,P_d \in \ZZ$ such that $|Q \alpha_i - P_i - \frac12| \leq \delta_d$, for all $1 \leq i \leq d$.
We claim that the lattice point
\[
\bw' = Q \bw - \sum_{i=1}^d P_i \bv_i = \sum_{i=1}^d (Q \alpha_i - P_i) \bv_i
\]
lies in the interior of~$P$ as well and satisfies $\ca(P,\bw') \leq \frac{\frac12 + \delta_d}{\frac12 - \delta_d}$.
The former holds, since we must have $0 < Q \alpha_i - P_i < 1$, because $\delta_d < \frac12$.
The latter holds by Corollary~\ref{cor:ca-parallelepipeds}, as $\max_{1 \leq i \leq d}|Q \alpha_i - P_i - \frac12| \leq \delta_d$.
\end{proof}

Inspired by the examples in Proposition~\ref{prop:psi-d-extremal} that attain~$w(d)$, we construct, for each integer $z \geq 3$, a lattice parallelepiped, all of whose interior lattice points have coefficient of asymmetry equal to~$z-1$.

\begin{proposition}
\label{prop:extremal-parallelepipeds}
For an integer $z \geq 3$, we write $k=\frac12\phi(z)$ and $d(z) = k + h(z)$.
Let $\alpha^z \in (\QQ \setminus \ZZ)^{d(z)}$ be defined as in Proposition~\ref{prop:psi-d-extremal} and let $P^z = [\bzero,\bv^z] + [\bzero,\be_2] + \dots + [\bzero,\be_{d(z)}] \subseteq \RR^{d(z)}$ be the lattice parallelepiped whose first generator is given by
\[
\bv^z := \left( z, \, z \, \alpha_2^z, \, \ldots, \, z \, \alpha_k^z, \, z \,
\alpha_{k+1}^z, \, \ldots, \, z \, \alpha_{k+h(z)}^z \right)^\intercal \in \ZZ^{d(z)}.
\]
There are $\phi(z)$ interior lattice points in~$P^z$ and each such $\bw \in \inter(P^z) \cap \ZZ^{d(z)}$ satisfies $\ca(P^z,\bw) = z-1$.
\end{proposition}

\begin{proof}
Let $\bw = (w_1,w_2,\ldots,w_{d(z)})^\intercal \in \inter(P^z) \cap \ZZ^{d(z)}$ be an interior lattice point.
By construction, $a := w_1 \in \{1,2,\ldots,z-1\}$, and there are uniquely determined coefficients $\beta_2,\ldots,\beta_{d(z)} \in (0,1)$ such that
\begin{align}
\bw &= \frac{a}{z} \bv^z + \beta_2 \be_2 + \dots + \beta_{d(z)} \be_{d(z)} = \left(a , \beta_2 + a \alpha^z_2,\ldots,\beta_{d(z)} + a \alpha^z_{d(z)}\right)^\intercal.\label{eqn:w-representation}
\end{align}
Now, if $\gcd(a,z) \geq 2$, then $h(z) \geq 1$ as~$z$ cannot be prime, and one of the prime factors $p_1,\ldots,p_{h(z)}$ of~$z$, say~$p_j$, divides~$a$.
Since $\alpha^z_{k+j} = \frac{1}{p_j}$, this means that $a \alpha^z_{k+j} \in \ZZ$ and thus $\beta_{k+j} \in \ZZ$, contradicting that all $\beta_i \in (0,1)$.
Therefore, we necessarily have $\gcd(a,z) = 1$.

On the other hand, for a given $a \in \{1,2,\ldots,z-1\}$ with $\gcd(a,z) = 1$, there is a unique interior lattice point $\bw$ in~$P^z$.
In fact, by~\eqref{eqn:w-representation} the only choice of the coefficients $\beta_j$ to
make~$\bw$ integral is $\beta_j = 1 - \{a \alpha^z_j\}$, where $\{\dots\}$ denotes the fractional part.
By assumption on~$a$ and the definition of~$\alpha^z$, we obtain $0 < \beta_j < 1$, as desired.

We now argue that each interior lattice point $\bw \in \inter(P^z) \cap \ZZ^{d(z)}$ satisfies $\ca(P^z,\bw) = z-1$.
In view of Corollary~\ref{cor:ca-parallelepipeds}, $\ca(P^z,\bw)$ corresponds to $\max\left\{|\frac{a}{z}-\frac12| , \max_{2 \leq i \leq d(z)} |\beta_i - \frac12|\right\}$.
We need to show that this number always equals $\frac12-\frac{1}{z}$.
Using the representation~\eqref{eqn:w-representation} for~$\bw$ again, we first see that this maximum is at most $\frac12-\frac{1}{z}$, because by the definition of~$\alpha^z$, we have $\beta_i = 1 - \{a \alpha^z_i\} \in (\frac{1}{z} \ZZ) \setminus \ZZ$, for $2 \leq i \leq d(z)$.
On the other hand, as~$a$ and $z$ are coprime, there is an integer $r \in \ZZ$ such that $ar \equiv 1$ mod~$z$.
Recall that the first $k=\frac12 \phi(z)$ entries of $\alpha^z$ are given by $\alpha^z_i = \frac{a_i}{z}$, $1 \leq i \leq k$, and where $1 \leq a_1 < \dots < a_k \leq \frac{z}{2}$ constitute the first half of coprime integers to~$z$.
This means that there is an index $1 \leq i \leq k$ such that $r \equiv \pm a_i$ mod~$z$, that is, $a_i = \pm r + \ell z$, for some $\ell \in \ZZ$.
Thus
\[
\left\{a \alpha^z_i\right\} = \left\{\frac{aa_i}{z}\right\} = \left\{\pm\frac{a r}{z} + \ell\right\} = \left\{\pm\frac{ar}{z}\right\} = \left\{\pm\frac{1}{z}\right\}.
\]
For the corresponding coefficient $\beta_i$ this means that $|\beta_i-\frac12| = |1 - \left\{a \alpha^z_i\right\} - \frac12| = \frac12-\frac{1}{z}$.
\end{proof}

Some data on the parallelepipeds $P^z$ from Proposition~\ref{prop:extremal-parallelepipeds} is collected in Table~\ref{tbl:extremal-parallelepipeds}, and Figure~\ref{fig:P4-P5} shows the particular instances~$P^4$ and~$P^5$ and their interior lattice points.

\begin{figure}
\hfill
\includegraphics[scale=.9]{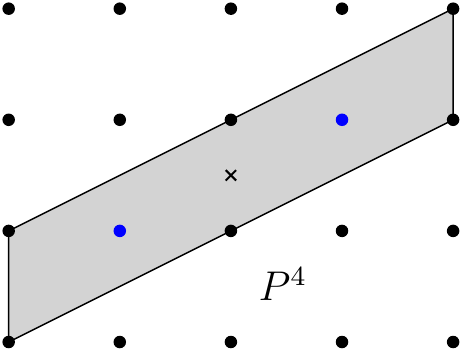}
\hfill
\includegraphics[scale=.9]{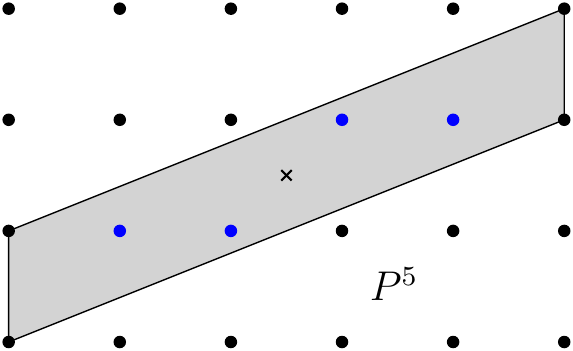}
\hfill\,
\caption{The lattice parallelograms $P^4$ and $P^5$ and their interior lattice points.}
\label{fig:P4-P5}
\end{figure}

\begin{table}[ht]
\centering
\captionsetup{position=below}
\begin{tabular}{|c|c|c|c|c|c|}
\hline
\textbf{$z$} & \textbf{$k = \frac12 \varphi(z)$} & \textbf{$h(z)$} & \textbf{$\alpha^z$} & \textbf{$\bv^z$} & \textbf{$\dim(P^z) = d(z)$} \\ \hline
$3$ & $1$ & $0$ & $(\frac13)$ & $(3)$ & $1$  \\ \hline
$4$ & $1$ & $1$ & $(\frac14,\frac12)$ & $(4,2)$ & $2$ \\ \hline
$5$ & $2$ & $0$ & $(\frac15,\frac25)$ & $(5,2)$ & $2$ \\ \hline
$6$ & $1$ & $2$ & $(\frac16,\frac13,\frac12)$ & $(6,2,3)$ & $3$ \\ \hline
$7$ & $3$ & $0$ & $(\frac17,\frac27,\frac37)$ & $(7,2,3)$ & $3$ \\ \hline
$8$ & $2$ & $1$ & $(\frac18,\frac38,\frac12)$ & $(8,3,4)$ & $3$ \\ \hline
$9$ & $3$ & $1$ & $(\frac19,\frac29,\frac49,\frac13)$ & $(9,2,4,3)$ & $4$ \\ \hline
$10$ & $2$ & $2$ & $(\frac{1}{10},\frac{3}{10},\frac12,\frac15)$ & $(10,3,5,2)$ & $4$ \\ \hline
$11$ & $5$ & $0$ & $(\frac{1}{11},\frac{2}{11},\frac{3}{11},\frac{4}{11},\frac{5}{11})$ & $(11,2,3,4,5)$ & $5$ \\ \hline
$12$ & $2$ & $2$ & $(\frac{1}{12},\frac{5}{12},\frac12,\frac13)$ & $(12,5,6,4)$ & $4$ \\ \hline
$13$ & $6$ & $0$ & $(\frac{1}{13},\frac{2}{13},\frac{3}{13},\frac{4}{13},\frac{5}{13},\frac{6}{13})$ & $(13,2,3,4,5,6)$ & $6$ \\ \hline
$14$ & $3$ & $2$ & $(\frac{1}{14},\frac{3}{14},\frac{5}{14},\frac12,\frac17)$ & $(14,3,5,7,2)$ & $5$ \\ \hline
$15$ & $4$ & $2$ & $(\frac{1}{15},\frac{2}{15},\frac{4}{15},\frac{7}{15},\frac13,\frac15)$ & $(15,2,4,7,5,3)$ & $6$ \\ \hline
$16$ & $4$ & $1$ & $(\frac{1}{16},\frac{3}{16},\frac{5}{16},\frac{7}{16},\frac12)$ & $(16,3,5,7,8)$ & $5$ \\ \hline
$17$ & $8$ & $0$ & $(\frac{1}{17},\frac{2}{17},\frac{3}{17},\frac{4}{17},\frac{5}{17},\frac{6}{17},\frac{7}{17},\frac{8}{17})$ & $(17,2,3,4,5,6,7,8)$ & $8$ \\ \hline
$18$ & $3$ & $2$ & $(\frac{1}{18},\frac{5}{18},\frac{7}{18},\frac12,\frac13)$ & $(18,5,7,9,6)$ & $5$ \\ \hline
$19$ & $9$ & $0$ & $(\frac{1}{19},\frac{2}{19},\frac{3}{19},\frac{4}{19},\frac{5}{19},\frac{6}{19},\frac{7}{19},\frac{8}{19},\frac{9}{19})$ & $(19,2,3,4,5,6,7,8,9)$ & $9$ \\ \hline
$20$ & $4$ & $2$ & $(\frac{1}{20},\frac{3}{20},\frac{7}{20},\frac{9}{20},\frac12,\frac15)$ & $(20,3,7,9,10,4)$ & $6$ \\ \hline
\end{tabular}
\caption{The definining data and the dimension of the parallelepipeds $P^z$, for $3 \leq z \leq 20$.}
\label{tbl:extremal-parallelepipeds}
\end{table}

Once we take care of a subtlety regarding the dimensions of~$P^z$, we can turn them into examples that show that the bound in Theorem~\ref{thm:ca-parallelepipeds} is best possible, for any given dimension~$d$.
This subtlety arises from the fact that there are integers~$d$ such that for the maximal $z \in \ZZ$ with $d(z) = \frac12\phi(z) + h(z) \leq d$, we have strict inequality $d(z) < d$.
For instance, this happens for the first time for $d=12$.
Here $w(12) = 60$, but $d(60) = \frac12\varphi(60) + h(60) = 11$ (see Table~\ref{tbl:psi-d-values}).
However, in such cases we may suitably lift $P^z$ into the correct dimension.
% by the number of dimensions that are missing.

\begin{proposition}
\label{prop:extremal-lattice-parallelepipeds}
Let $d \in \NN$ and let $z = w(d)$.
Then every interior lattice point~$\bw$ of the lattice parallelepiped
\[
P(d) := P^z + [\bzero,2 \be_{d(z)+1}] + \dots + [\bzero,2\be_d] \subseteq \RR^d
\]
satisfies $\ca(P(d),\bw) = w(d) - 1$, where we think of $P^z$ as embedded into the subspace $\RR^{d(z)} \times \{\bzero\}^{d-d(z)}$.

In particular, the bound in Theorem~\ref{thm:ca-parallelepipeds} is best possible.
\end{proposition}

\begin{proof}
Every interior lattice point $\bw \in \inter(P(d)) \cap \ZZ^d$ is of the form $\bw = (w_1',\ldots,w_{d(z)}',1,\ldots,1)$, where $\bw' = (w_1',\ldots,w_{d(z)}')$ is an interior lattice point of~$P^z$.
The coefficients in the representation of~$\bw$ in the basis of generators of~$P(d)$, which correspond to the last $d-d(z)$ coordinates are equal to~$\frac12$, independently of the particular choice of~$\bw$.
Thus, in view of Corollary~\ref{cor:ca-parallelepipeds} and Proposition~\ref{prop:extremal-parallelepipeds}, 
\[
\ca(P(d),\bw) = \ca(P^z,\bw') = z-1 = w(d) - 1 \, ,
\]
as desired.
\end{proof}

\begin{remark}
The parallelepipeds $P(d)$ in Proposition~\ref{prop:extremal-lattice-parallelepipeds} exhibit an even stronger extremality property.
Namely, one may ask whether one can improve and refine the bound in Theorem~\ref{thm:ca-parallelepipeds} in terms of the number $\ell = \#\left( \inter(P) \cap \ZZ^d \right) \geq 1$ of interior lattice points in a given lattice parallelepiped~$P \subseteq \RR^d$.
However, since~$P(d)$ has exactly $\phi(w(d))$ interior lattice points (see Proposition~\ref{prop:extremal-parallelepipeds})---a number that asymptotically grows to infinity with~$d$---such a refinement is not possible when we ask the parameter~$\ell$ to be independent of the dimension~$d$.
\end{remark}

% ##########################################################
\section{Deep Lattice Points in Arbitrary Lattice Zonotopes}
\label{sect:deeplattptszonotopes}
% ##########################################################

In this section, we extend the bound in Theorem~\ref{thm:ca-parallelepipeds} from lattice parallelepipeds to arbitrary lattice zonotopes, and thus complete the proof of Theorem~\ref{thm:main-ca-precise-bound-zonotopes}.
Our argument is based on a zonotopal version of the following theorem of Steinitz.

\begin{theorem}[Steinitz 1914]
Let $S \subseteq \RR^d$ and let~$\bw$ be a point in the interior of the convex hull of~$S$.
Then there are at most~$2d$ points in~$S$ that contain~$\bw$ in the interior of their convex hull.
\end{theorem}

Steinitz' work~\cite[\S 20]{steinitz1914bedingtII} is in the language of ``all-sided families of rays''; Gustin~\cite{gustin1947ontheinterior} gives a different proof in the formulation above, and we also refer to Gr\"unbaum~\cite{gruenbaum1962thedimension} for more pointers to incarnations of Steinitz' result in various other contexts.
Our zonotopal version now reads as follows:

\begin{theorem}
\label{thm:zonotopal-steinitz}
Let $Z = \Zon(\bg_1,\dots,\bg_m)$ be the zonotope generated by $\bg_1,\dots,\bg_m \in \RR^d$, and let~$\bc$ be its center.
For every $\bv \in \inter(Z)$, there is a $(d-1)$-dimensional parallelepiped~$Q$ that is a translate of a parallelepiped generated by a subset of $\{\bg_1,\dots,\bg_m\}$ such that~$\bv$ is contained in the interior of the parallelepiped $P \subseteq Z$ given by $P = \conv\left( Q, 2\bc - Q \right)$.

Moreover, if~$Z$ is a lattice zonotope, then~$P$ is a lattice parallelepiped.
\end{theorem}

\begin{proof}
Let $\bc = \frac12(\bg_1+\dots+\bg_m)$ be the center of~$Z$.
Consider the line~$\ell$ passing through~$\bc$ and~$\bv$.
In the special case that $\bv = \bc$, we let~$\ell$ be an arbitrary line through~$\bc$.
Now, $\ell$ intersects the boundary of~$Z$ in a pair of opposite points~$\bp$ and~$\bp'$.
Let $F$ be a facet of~$Z$ that contains~$p$.
Facets of zonotopes are zonotopes themselves, and in particular, we can write
\begin{align}
F = \Zon(\bg_{\sigma(1)},\dots,\bg_{\sigma(s)}) + \varepsilon_{s+1} \bg_{\sigma(s+1)} + \dots + \varepsilon_{m} \bg_{\sigma(m)},\label{eqn:zonotope-face-representation}
\end{align}
for some $s \in [m]$, some permutation $\sigma$ of~$[m]$, and some $\varepsilon_i \in \{0,1\}$ (see, e.g.,~\cite{shephard1974combinatorial}).

By a result of Shephard~\cite{shephard1974combinatorial} (see also~\cite[Chapter~7]{becksanyal2018combinatorial}) the zonotope~$F$ admits a tiling into parallelepipeds of the form~\eqref{eqn:zonotope-face-representation}, for suitable~$s$, $\sigma$ and~$\varepsilon_i$.
Upon taking~$Q$ to be a parallelepiped in such a tiling of~$F$ with $\bp \in Q$, we find that the interior of the segment with endpoints~$\bp$ and~$\bp'$ is contained in the interior of the $d$-dimensional parallelepiped $P = \conv\left( Q, 2\bc - Q \right) \subseteq Z$, and in particular $\bv \in \inter(P)$.

If~$Z$ is a lattice zonotope, every zonotope of the form~\eqref{eqn:zonotope-face-representation} is a lattice zonotope itself, in particular~$Q$ is.
Moreover, $2\bc \in \ZZ^d$ and thus $2\bc - Q$ is a lattice parallelepiped, implying that the constructed parallelepiped $P \subseteq Z$ is a lattice parallelepiped as well.
\end{proof}

\begin{proof}[Proof of Theorem~\ref{thm:main-ca-precise-bound-zonotopes}]
Let $\bv \in \inter(Z) \cap \ZZ^d$.
Then, by Theorem~\ref{thm:zonotopal-steinitz}, we find a lattice parallelepiped~$P \subseteq Z$ having the same center as~$Z$ and with $\bv \in \inter(P)$.
Applying Theorem~\ref{thm:ca-parallelepipeds} to~$P$, we obtain an interior lattice point $\bw \in \inter(P)$ with $\ca(P,\bw) \leq w(d) - 1$.
Since~$P$ and~$Z$ are symmetric about the same point, we can use the
translation invariance of the coefficient of asymmetry (see
Proposition~\ref{prop:ca-covariance-translation-scaling}) and obtain $\ca(Z,\bw) \leq \ca(P,\bw) \leq w(d) - 1$ by Corollary~\ref{cor:ca-monotonicity}.

The asymptotic behavior $w(d)-1 \in \Theta(d \log\log d)$ follows from the one of $\rab(d)$ in~\eqref{eqn:psi-d-asymptotics}.
Optimality of the derived bound has been established already in Proposition~\ref{prop:extremal-lattice-parallelepipeds}.
\end{proof}

% #######################################################
\section{Integer-point Enumeration for the Lonely Runner Zonotope}
\label{sect:ehrhart}
% #######################################################

In this part, we aim to give a geometric intuition for the ``finite checking'' result for the Lonely Runner Conjecture due to Tao~\cite{tao2018someremarks}.

\begin{theorem}[{Tao~\cite[Theorem~1.4]{tao2018someremarks}}]
\label{thm:tao-finite-checking-lrc}
There exists an absolute and explicitly computable constant $C_0 > 0$, such that the following statements are equivalent for every natural number $d_0 \geq 1$:
\begin{enumerate}[(i)]
 \item The Lonely Runner Conjecture holds for every dimension $d \leq d_0$.
 \item The bound $\gap(\bn) \geq \frac{1}{d+1}$ holds for every velocity vector $\bn \in \ZZ^d_{>0}$ with $d \leq d_0$ and $n_i \leq d^{C_0 d^2}$,  $i \in [d]$.
\end{enumerate}
\end{theorem}
We take the viewpoint of the lonely runner zonotope~$Z_{\bn} = \Zon(\be_1,\ldots,\be_d,\bn)$, that we introduced at the end of Section~\ref{sect:lonelyrunners}.
Based on the reformulation of the Lonely Runner Conjecture in that earlier section on the existence of deep interior lattice points in the zonohedron~$Z(\bn)$, the heuristic thought is as follows:
If there is a sufficiently large number of interior lattice points in any localized version of~$Z(\bn)$, then one of those lattice points lies deep enough in~$Z(\bn)$ to satisfy the conjectured bound.

It turns out that the number of (interior) lattice points in~$Z_\bn$ can be calculated explicitly.

\begin{theorem}
\label{thm:num-ilps-lr-zonotopes}
Let $\bn \in \ZZ_{>0}^d$ be a velocity vector.
Then $\vol(Z_\bn) = 1 + n_1 + \dots + n_d$ and
\[
\#(\inter(Z_{\bn}) \cap \ZZ^d) \ = \sum_{\substack{\ell \in \NN \textrm{ such that}\\
\ell \mid n_j \textrm{ for some } j \in [d]}} \varphi(\ell)
\qquad \textrm{and} \qquad \#(Z_\bn \cap \ZZ^d) = 2^d + \sum_{\substack{\ell \in \NN \textrm{ such that}\\
\ell \mid n_j \textrm{ for some } j \in [d]}} \varphi(\ell) \left(2^{\# J_\ell} - 1\right) \,,
\]
where $J_\ell$ is the inclusion-wise maximal subset $J \subseteq [d]$ such that $\ell \mid n_j$, for all $j \in J$.
\end{theorem}

\begin{proof}
We write
\[
Z_{\bn} = \sum_{j=1}^{d+1} [\bzero,\bv_j], \quad \textrm{ where } \quad V =
(\bv_1,\ldots,\bv_{d+1}) = (\be_1,\ldots,\be_d,\bn) \, .
\]
We now employ Ehrhart theory; see, e.g.,~\cite{becksanyal2018combinatorial}. The
lattice-point counting function $\#(k Z_\bn \cap \ZZ^d)$ is a polynomial in the positive integer variable $k$, which we write as
\[
  \#(k Z_\bn \cap \ZZ^d) = \sum_{ j=0 }^{ d } g_j(Z_{\bn}) \, k^j .
\]
Since $n_i \neq 0$, for every $i \in [d]$, Stanley's formula (see, e.g.,~\cite[Ex.~31,
p.~272]{stanley1997enumerative}) for the Ehrhart coefficients $g_i(Z_{\bn})$ of~$Z_{\bn}$ yields
\begin{align}
g_i(Z_{\bn}) = \sum_{J \in \binom{[d+1]}{i}} \gcd\left(i\textrm{-minors of }V_J\right) = \binom{d}{i} + \sum_{J \in \binom{[d]}{i-1}} \gcd\left( n_j : j \notin J \right) \,,\label{eqn:ehrhart-Z-n}
\end{align}
where $V_J$ denotes the subset of $V$ indexed by $J$.
In terms of the Ehrhart coefficients, the three functionals we are after have the expressions~\cite[Chapter~5]{becksanyal2018combinatorial}
\[
\vol(Z_\bn) = g_d(Z_\bn), \quad \#(\inter(Z_\bn) \cap \ZZ^d) = (-1)^d \sum_{i=0}^d (-1)^ig_i(Z_{\bn}) \quad \textrm{and} \quad \#(Z_\bn \cap \ZZ^d) = \sum_{i=0}^d g_i(Z_{\bn}) \,.
\]
The formula $\vol(Z_\bn) = 1 + n_1 + \dots + n_d$ thus follows directly from~\eqref{eqn:ehrhart-Z-n}, while we have to work a little more for the lattice point counts.
We first consider the number of interior lattice points of~$Z_\bn$:
\begin{align}
\#(\inter(Z_{\bn}) \cap \ZZ^d) &= (-1)^d \sum_{i=0}^d (-1)^ig_i(Z_{\bn}) = (-1)^d \sum_{i=0}^d(-1)^i \binom{d}{i} + (-1)^d \sum_{i=0}^d(-1)^i\sum_{J \in \binom{[d]}{i-1}} \gcd\left( n_j : j \notin J \right) \nonumber\\
&= \sum_{i=1}^d(-1)^{d+i}\sum_{J \in \binom{[d]}{i-1}} \gcd\left( n_j : j \notin J \right) = \sum_{i=1}^d(-1)^{i+1}\sum_{J \in \binom{[d]}{i}} \gcd\left( n_j : j \in J \right) \,.\label{eqn:num-ilps-lr-zonotope}
\end{align}
Now, we employ Gau\ss' identity $m = \sum_{\ell \mid m} \phi(\ell)$ for the totient function, where the sum runs over
all divisors~$\ell$ of a given positive integer~$m$.
In order to apply it for $m = \gcd(n_j : j \in J)$, for some $ J \subseteq [d]$, we use
the notation $\ell \mid n_J$ for the statement that $\ell \mid n_j$ for every $j \in J$.
Also, writing $\gcd(n_J) := \gcd(n_j : j \in J)$, we obtain 
\[
\gcd(n_J) = \sum_{\ell \mid \gcd(n_J)} \phi(\ell) = \sum_{\ell \mid n_J} \phi(\ell) \, .
\]
With~\eqref{eqn:num-ilps-lr-zonotope} this gives
\begin{align*}
\#(\inter(Z_{\bn}) \cap \ZZ^d) &= \sum_{i=1}^d(-1)^{i+1}\sum_{J \in \binom{[d]}{i}} \gcd\left( n_J \right) = \sum_{i=1}^d(-1)^{i+1}\sum_{J \in \binom{[d]}{i}} \sum_{\ell \mid n_J} \phi(\ell) \\
&= \sum_{i=1}^d(-1)^{i+1} \sum_{\ell = 1}^{\max(\bn)} \phi(\ell) \cdot \#\left\{J \in \binom{[d]}{i} : \ell \mid n_J \right\} \\
&= \sum_{\ell = 1}^{\max(\bn)} \phi(\ell) \, \underbrace{\sum_{i=1}^d(-1)^{i+1} \cdot
\#\left\{J \in \binom{[d]}{i} : \ell \mid n_J \right\}}_{=: s_{\ell}(\bn)}  \, .
\end{align*}
In order to proceed, we let $D_\ell(\bn) := \left\{ i \in [d] : \ell \mid n_i \right\}$.
Clearly, if none of the coordinates of~$\bn$ is divisible by~$\ell$, that is, $D_\ell(\bn) = \emptyset$, we get $s_\ell(\bn) = 0$.
If $D_\ell(\bn) \neq \emptyset$, then
\[
s_\ell(\bn) = \sum_{i=1}^{\# D_\ell(\bn)} (-1)^{i+1} \cdot \#\left\{J \in \binom{D_\ell(\bn)}{i} : \ell \mid n_J \right\} = \sum_{i=1}^{\# D_\ell(\bn)} (-1)^{i+1} \binom{\# D_\ell(\bn)}{i} = 1
\]
and we arrive at
\[
\#(\inter(Z_{\bn}) \cap \ZZ^d) = \sum_{\ell = 1}^{\max(\bn)} \phi(\ell) \, s_\ell(\bn)
= \sum_{\substack{\ell = 1\\ D_\ell(\bn) \neq \emptyset}}^{\max(\bn)} \phi(\ell) \,
s_\ell(\bn) = \sum_{\substack{\ell \in \NN \textrm{ such that}\\ \ell \mid n_j \textrm{
for some } j \in [d]}} \varphi(\ell) \, ,
\]
as desired.

For the total number of lattice points in~$Z_\bn$ we argue similarly and observe that in view of~\eqref{eqn:ehrhart-Z-n}
\begin{align*}
\#(Z_{\bn} \cap \ZZ^d) &= \sum_{i=0}^d g_i(Z_{\bn}) = 2^d + \sum_{i=0}^d \sum_{J \in \binom{[d]}{i-1}} \gcd\left( n_j : j \notin J \right) = 2^d + \sum_{i=1}^d \sum_{J \in \binom{[d]}{i}} \gcd(n_J) \\
&= 2^d + \sum_{i=1}^d \sum_{J \in \binom{[d]}{i}} \sum_{\ell \mid n_J} \varphi(\ell) = 2^d + \sum_{\ell=1}^{\max(\bn)} \varphi(\ell) \cdot \#\left\{ J \subseteq [d] : J \neq \emptyset, \ell \mid n_J \right\} \\
&= 2^d + \sum_{\substack{\ell \in \NN \textrm{ such that}\\
\ell \mid n_j \textrm{ for some } j \in [d]}} \varphi(\ell) \left(2^{\# J_\ell} - 1\right) \,,
\end{align*}
as claimed.
\end{proof}

Let's focus for a moment on the supposedly extremal instance $\lext{d} = (1,2,\ldots,d)$ for the lonely runner problem.
By the formula in Theorem~\ref{thm:num-ilps-lr-zonotopes}, we obtain $\#(\inter(Z_{\lext{d}}) \cap \ZZ^d) = \sum_{\ell=1}^d \varphi(\ell) =: \Phi(d)$ -- the \emph{totient summatory function}.
The first few values of $\Phi(d)$ are given in Table~\ref{tbl:Phi-d-values} and its asymptotic behavior is given by $\Phi(d) = \frac{3}{\pi^2} d^2 + \cO(d \log{d})$ (see, e.g.,~\cite[Section~18.5]{hardywright1979anintroduction}).
\begin{table}[ht]
\centering
\captionsetup{position=below}
\begin{tabular}{c|c|c|c|c|c|c|c|c|c|c|c|c|c|c|c}
$d$      & $1$           & $2$           & $3$           & $4$            & $5$            & $6$            & $7$            & $8$            & $9$            & $10$           & $11$           & $12$           & $13$ & $14$ & $15$           \\ \hline
$\Phi(d)$ & $1$ & $2$ & $4$ & $6$ & $10$ & $12$ & $18$ & $22$ & $28$ & $32$ & $42$ & $46$ & $58$ & $64$ & $72$
\end{tabular}
\caption{The first few values of the totient summatory function $\Phi(d)$.}
\label{tbl:Phi-d-values}
\end{table}

\begin{corollary}
\label{cor:bounds-nint-latpoints-lrc-zonotope}
Let $\bn \in \ZZ_{>0}^d$ be a velocity vector.
Then
\[
\max_{1 \leq i \leq d} n_i \leq \#(\inter(Z_{\bn}) \cap \ZZ^d) \leq \sum_{i=1}^d n_i.
\]
Moreover, if $\displaystyle\max_{1 \leq i \leq d} \,n_i \geq \Phi(d)$, then $\#(\inter(Z_{\bn}) \cap
\ZZ^d) \geq \#(\inter(Z_{\lext{d}}) \cap \ZZ^d)$, where $\lext{d} = (1,2,\ldots,d)$.
\end{corollary}

\begin{proof}
By Theorem~\ref{thm:num-ilps-lr-zonotopes}, we have $\#(\inter(Z_{\bn}) \cap \ZZ^d) = \sum_{\ell \mid n_j \textrm{ for some } j \in [d]} \varphi(\ell)$.
Thus, for every $i \in [d]$, Gau\ss' identity gives us $\#(\inter(Z_{\bn}) \cap \ZZ^d) \geq \sum_{\ell \mid n_i} \varphi(\ell) = n_i$.
Likewise, we obtain
\[
\#(\inter(Z_{\bn}) \cap \ZZ^d) \leq \sum_{i=1}^d \sum_{\ell \mid n_i} \varphi(\ell) =
\sum_{i=1}^d n_i \, . \qedhere
\]
\end{proof}

The corollary details the geometric heuristic for Tao's result described in the beginning of this section.
In fact, by virtue of the second part, if any of the  velocities~$n_i$ exceeds~$\Phi(d)$, then there are at least as many interior lattice points in~$Z_\bn$ as in the supposedly extremal instance~$Z_{\lext{d}}$.

\begin{remark}
\label{rem:other-extremal-lrc-instances}
The velocity vector $\lext{d}$ is not the only supposedly extremal case for the Lonely Runner Conjecture.
Goddyn \& Wong~\cite{goddynwong2006tight} identified numerous non-canonical extremal examples, for instance, $(1,3,4,7)$ and $(1,3,4,5,9)$, and also described how to obtain such an example, for any given~$d$, by carefully modifying~$\lext{d}$.
Regarding the number $\#(\inter(Z_{\bn}) \cap \ZZ^d)$ for any of those extremal instances~$\bn$, the canonical velocity vector~$\lext{d}$ always is the unique such instance with the minimal number of interior lattice points.
In this sense the count of interior lattice points of~$Z_\bn$ distinguishes between those extremal instances.
\end{remark}

For $2 \leq d \leq 5$, the $\Phi(d)$ interior lattice points in $Z_{\lext{d}}$ are given by the columns of the following matrices, respectively.
The coefficient of asymmetry of the respective column is written in the very last additional row:
\[
\left(\begin{matrix}
1 & 1 \\
1 & 2 \\ \hline
2 & 2
\end{matrix}\right),
\left(\begin{matrix}
1 & 1 & 1 & 1 \\
1 & 1 & 2 & 2 \\
1 & 2 & 2 & 3 \\ \hline
3 & 4 & 4 & 3
\end{matrix}\right),
\left(\begin{matrix}
1 & 1 & 1 & 1 & 1 & 1\\
1 & 1 & 1 & 2 & 2 & 2 \\
1 & 1 & 2 & 2 & 3 & 3 \\
1 & 2 & 2 & 3 & 3 & 4 \\ \hline
4 & 6 & 4 & 4 & 6 & 4
\end{matrix}\right),
\left(\begin{matrix}
1 & 1 & 1 & 1 & 1 & 1 & 1 & 1 & 1 & 1 \\
1 & 1 & 1 & 1 & 1 & 2 & 2 & 2 & 2 & 2 \\
1 & 1 & 1 & 2 & 2 & 2 & 2 & 3 & 3 & 3 \\
1 & 1 & 2 & 2 & 2 & 3 & 3 & 3 & 4 & 4 \\
1 & 2 & 2 & 2 & 3 & 3 & 4 & 4 & 4 & 5 \\ \hline
5 & 8 & 6 & 7 & 6 & 6 & 7 & 6 & 8 & 5
\end{matrix}\right).
\]
Since $\lext{d}$ is an extremal velocity vector, we need to have $\ca(Z_{\lext{d}},\bw) \geq d$, for every interior lattice point $\bw \in \inter(Z_{\lext{d}})$.
As the data in small dimensions suggest, this bound is in fact attained by $\bw = \bone \in \inter(Z_{\lext{d}})$ in any dimension~$d$.

\begin{proposition}
For every $d \in \NN$, we have $\ca(Z_{\lext{d}},\bone) = d$.
\end{proposition}

\begin{proof}
In view of Proposition~\ref{prop:ca-versus-minkowski-norm} it suffices to prove that $\|\bone - \bc\|_{Z_{\lext{d}} - \bc} = \frac{d-1}{d+1}$, which is equivalent to showing that the point $\bc + \frac{d+1}{d-1}(\bone - \bc)$ lies in the boundary of~$Z_{\lext{d}}$, where $\bc = \frac12(\bone + \lext{d})$ is the center of~$Z_{\lext{d}}$.
An elementary computation yields
\[
\bc + \frac{d+1}{d-1}(\bone - \bc) = \left( 1, \frac{d-2}{d-1},\ldots,\frac{1}{d-1},0 \right)^\intercal,
\]
from which the containment in the boundary of~$Z_{\lext{d}}$ is clear.
\end{proof}

% #############################
\section{Outlook and Questions}
\label{sect:outlook}
% #############################

Our bound in Theorem~\ref{thm:main-ca-precise-bound-zonotopes} on the coefficient of asymmetry concerns arbitrary lattice zonotopes.
For the Lonely Runner Conjecture these are too general, while from the viewpoint of the class of symmetric lattice polytopes they are too special.
As an outlook for future work, we discuss an open question in each of the two settings.

\subsection{Cubical lattice zonotopes}

A zonotope is called \emph{cubical} if each of its facets (and thus each of its proper faces) is a parallelepiped.
Equivalently, a $d$-dimensional zonotope is cubical if and only if every~$d$ of its
generators are linearly independent (see, e.g.,~\cite{shephard1974combinatorial}).

Given a velocity vector $\bn \in \ZZ^d_{>0}$, the lonely runner zonohedron $Z(\bn)$ projects along~$\RR \bn$ onto a cubical $(d-1)$-dimensional zonotope with~$d$ generators.
It is a lattice zonotope with respect to the projected lattice $\ZZ^d | \bn^\perp$, and by Proposition~\ref{prop:ca-linear-recession-cone}, we have $\ca(Z(\bn),\bw) = \ca(Z(\bn)|\bn^\perp,\bw|\bn^\perp)$, for every interior point $\bw \in \inter(Z(\bn))$.
In terms of~\eqref{eqn:ca-gap-relation} and Wills' bound $\gap(\bn) \geq \frac{1}{2d}$ \cite{willslonelyrunner}, we always find an interior lattice point $\bw \in \inter(Z(\bn))$ such that $\ca(Z(\bn),\bw) \leq 2d-1$.
Compared to the bound in Theorem~\ref{thm:main-ca-precise-bound-zonotopes} this is a \emph{linear} bound in the dimension~$d$.

The lower bound of Wills on $\gap(\bn)$ has been gradually improved over time, but even the best-known bound $\gap(\bn) \geq \frac{1}{2d} + \frac{c \log(d)}{d^2 (\log \log d)^2}$, for some $c > 0$, due to Tao~\cite{tao2018someremarks} does not reduce the factor~$2$ in~$\frac{1}{2d}$.
Thus, a positive answer to the following question would be great progress on the lonely runner problem:

\begin{question}
Is there a universal constant $\gamma < 2$ such that every cubical lattice zonotope $Z \subseteq \RR^d$ with $d+1$ generators contains an interior lattice point~$\bw$ with $\ca(Z,\bw) \leq \gamma \, d$ ?
\end{question}

\subsection{Arbitrary symmetric lattice polytopes}

The slightly superlinear upper bound in Theorem~\ref{thm:main-ca-precise-bound-zonotopes} for lattice zonotopes is in stark contrast to the necessarily double exponential upper bound in Theorem~\ref{thm:ca-general-lattice-polytopes} on arbitrary lattice polytopes.
One may wonder where the family of symmetric lattice polytopes is situated between these two extremes.
We think of symmetric lattice polytopes here as polytopes $P \subseteq \RR^d$ with $P - \bc = \bc - P$, for some not necessarily integral point $\bc \in \RR^d$.
As for lattice zonotopes, the center of~$P$ is necessarily half-integral, meaning $\bc \in \tfrac12\ZZ^d$.

\begin{question}
\label{qu:symmetric-lattice-polytopes}
What is the best possible constant~$\nu_d$ such that every symmetric lattice polytope $P \subseteq \RR^d$ with $\inter(P) \cap \ZZ^d \neq \emptyset$ contains an interior lattice point $\bw \in \inter(P)$ such that $\ca(P,\bw) \leq \nu_d$ ?
Can we at least determine the asymptotic behavior of~$\nu_d$?
\end{question}

Using Carath\'{e}odory's Theorem in the boundary of~$P$, one can argue similarly as in Theorem~\ref{thm:zonotopal-steinitz} and obtain a specialized version of Steinitz' result for symmetric lattice polytopes, in which the~$2d$ points span a lattice crosspolytope in~$P$.
This reduces the question above to the case of lattice crosspolytopes.
Just like Theorem~\ref{thm:ca-parallelepipeds} connects the lonely rabbit problem with lattice parallelepipeds that have interior lattice points, the constant~$\nu_d$ in Question~\ref{qu:symmetric-lattice-polytopes} connects to a Diophantine approximation problem behind crosspolytopes.
More precisely, we find that $\nu_d \leq \frac{1 + \kappa_d}{1 - \kappa_d}$, where
\[
\kappa_d := \sup_{\substack{\alpha \in \RR^d\\ \|\alpha\|_1 < 1}} \,\,\,\inf_{\bP \in \ZZ^d} \,\,\,\min_{\varepsilon \in \{-1,1\}^d} \| \alpha - (\varepsilon^\intercal \bP) \alpha - \bP \|_1 \,.
\]
% \matthias{for the derivation see Lemma 32 in lattice-zonotopes.pdf}
Here, $\|\bv\|_1 = \sum_{i=1}^d |v_i|$ denotes the $\ell_1$-norm of~$\bv \in \RR^d$.
For $d \leq 2$, the class of lattice parallelepipeds agrees with the class of lattice
crosspolytopes, so it is only consistent that one finds $\kappa_1 = 2 \delta_2 = \frac13$
and $\kappa_2 = 2 \delta_2 = \frac35$ (see~Theorem~\ref{thm:ca-parallelepipeds}).
However, we do not know the exact value of any $\kappa_d$, with $d \geq 3$, nor do we know the asymptotic behavior of~$\kappa_d$ as a function of~$d$.

%-- possibly identifying examples that show superpolynomiality of the best bound??

% --------------------

% \newpage
\bibliographystyle{amsplain}
\bibliography{mybib}

\end{document}